\documentclass[12pt]{amsart}
\usepackage{amssymb}
\usepackage{amsmath, amscd}
\usepackage{amsthm}
\usepackage{xcolor}

\usepackage[english]{babel}
\usepackage{graphicx}
\usepackage{epic}

\numberwithin{equation}{section}

\newtheorem{definition}{Definition}[section]
\newtheorem{lemma}[definition]{Lemma}
\newtheorem{theorem}[definition]{Theorem}
\newtheorem{corollary}[definition]{Corollary}

\newtheorem{proposition}[definition]{Proposition}

\newtheorem{em-example}[definition]{Example}
\newtheorem{em-def}[definition]{Definition}
\newtheorem{em-remark}[definition]{Remark}
\newtheorem{em-question}[definition]{Question}

\newenvironment{remark}{\begin{em-remark} \em }{\end{em-remark}}

\newcommand{\Q}{\mathbb Q}
\newcommand{\Z}{\mathbb Z}
\newcommand{\qc}{\mathbb{Z}(p^{\infty})}

\newcommand{\X}{\mathcal{X}}

\newcommand{\B}{\mathcal{B}}
\newcommand{\EB}{\mathcal{EB}}
\newcommand{\RH}{\mathcal{RH}}
\newcommand{\WH}{\mathcal{WH}}
\newcommand{\DF}{\mathcal{DF}}

\def\Ker{\mathrm{Ker}}

\def\P{\mathbb{P}}

\begin{document}

\author[P. Danchev]{Peter Danchev}
\address{Institute of Mathematics and Informatics, Bulgarian Academy of Sciences, 1113 Sofia, Bulgaria}
\email{danchev@math.bas.bg; pvdanchev@yahoo.com}
\author[B. Goldsmith]{Brendan Goldsmith}
\address{Technological University, Dublin, Dublin 7, Ireland}
\email{brendan.goldsmith@TUDublin.ie; brendangoldsmith49@gmail.com}
\author[F. Karimi]{Fatemeh Karimi}
\address{Department of Mathematics, Payame Noor University, Tehran, Iran}
\email{fatemehkarimi@pnu.ac.ir}
	
\vskip1.0pc

\title[Hereditary and Super Properties in Directly Finite Abelian Groups]{On Some Hereditary and Super Classes Of \\ Directly Finite Abelian Groups}
	
\begin{abstract} Continuing recent studies of both the hereditary and super properties of certain classes of Abelian groups, we explore in-depth what is the situation in the quite large class consisting of directly finite Abelian groups.

Trying to connect some of these classes, we specifically succeeded to prove the surprising criteria that a relatively Hopfian group is hereditarily only when it is extended Bassian, as well as that, a relatively Hopfian group is super only when it is extended Bassian. In this aspect, additional relevant necessary and sufficient conditions in a slightly more general context are also proved.
\end{abstract}

\maketitle

\medskip

\centerline{(Dedicated to the {\bf 70}th birthday of Patrick Warren Keef)}

\vskip1.0pc

\section{Introduction and Motivation}

The important notion of a Bassian group has its origins in works of Hyman Bass (see \cite{B1} and \cite{B2}), and was formally defined in \cite{CDG} as follows: the Abelian group $G$ is {\it Bassian} if $G$ cannot be embedded in a proper homomorphic image of itself. Part of the interest in such groups derives from the fact that they are Hopfian and that they are indeed much simpler than Hopfian groups in general; exactly, it is well known that the latter form a proper class in the set-theoretic sense: for any infinite cardinal $\kappa$, there exists a Hopfian group of cardinality $\geq \kappa$. On the other hand, Bassian groups are countable (cf. \cite{CDG}). Since the publication of \cite{CDG}, there has been considerable interest in other generalizations of the Bassian property (see, for example, \cite{CDG1}, \cite{CDGK}, \cite{DG}, \cite{DK}, \cite{K} and \cite{Ke} for a cross-section of these developments).

Another approach to simplifying the notion of Hopficity goes back to Hirshon's work on non-commutative groups (see \cite{H}): a group $G$ is said to be {\it super Hopfian} if every homomorphic image of $G$ is Hopfian. An obvious variation of this approach is to consider groups having the property that every subgroup is Hopfian; such groups are hereafter referred to as {\it hereditarily Hopfian} groups. If $\X$ is any class of groups, we respectively write $\X^s$ and $\X_h$ for the corresponding super and hereditarily $\X$ classes. In the case of Abelian groups, such super super and hereditarily Hopfian groups were classified in \cite{GG}. In our subsequent discussions all groups under consideration shall be additively written Abelian groups.

Considerable interest has also been shown to enlarging the class of Hopfian groups and in this context, a particularly interesting class is the class of {\it directly finite} groups, i.e., groups which do not have an isomorphic direct summand. This class is the \lq negation\rq \ of the class of $ID$ groups firstly studied by Beaumont and Pierce in \cite{BP}. In this paper, we focus attention on the super and hereditary properties in connection with the class of directly finite groups, denoted by $\DF$, and some of its subclasses.

In the course of our investigations, we found it necessary to introduce a further class of groups which turned out to be the key to unlocking the relationships between the various classes. We denote the class of Bassian groups by $\B$, and we define a group $G$ to be {\it Extended Bassian}, and write $G \in \EB$, if $G$ has the form $G = B \oplus D$, where $B$ is Bassian and $D$ is a torsion divisible group of the form $D = \bigoplus\limits_p\qc^{(n_p)}$ with each $n_p$ is finite (see Section \ref{hdfg}).

Our principal results, that are Theorem \ref{classher}, Theorem \ref{bxhcase} and Theorem \ref{xebscase}, sound thus:

\medskip

$\bullet$ if $\X$ is any subclass of Hopfian groups which contains the class $\B$ of Bassian groups, then a group $G$ is in $\X^s$ (resp., $G$ is in $\X_h$) if, and only if, $G$ is in $\B^s$ (resp., $G$ is in $\B$);

\medskip

$\bullet$ if $\X$ is any subclass of directly finite groups which contains the class $\EB$, then a group $G$ is in $\X^s$ (resp., $G$ is $\X_h$) if, and only if, $G$ is in $\EB$ (resp., $G$ is in $\EB$).

\medskip

In Section \ref{prelim}, we have taken the opportunity to present the original proof of an unpublished result of the late Tony Corner, which was useful in the current context, but may be of wider interest as well.

Finally, our notation is standard and may be found in the classical texts by Fuchs (cf. \cite{F1,F}).

\section{Classes of Directly Finite Groups}

Let $\mathcal{X}$ be a class of groups (we will always use class to mean that, if $G$ is a group belonging to the class, then any group isomorphic to $G$ also belongs to the class) and let $\mathcal{X}^s, \mathcal{X}_h$ denote the subclasses of $\mathcal{X}$ defined by

\medskip

(i) $\mathcal{X}^s = \{ G \in \mathcal{X}| \ {\rm every \ epimorphic \ image \ of} \ G \ {\rm is \ in} \ \mathcal{X} \}$

\medskip

(ii) $\mathcal{X}_h = \{ G \in \mathcal{X}| \ {\rm every \ subgroup \ of} \ G \ {\rm is \ in} \ \mathcal{X} \}$.

\medskip

Groups in $\X ^s$ are often called {\it super $\X$ groups}; those in $\X _h$ are referred to as {\it hereditarily $\X$ groups}.

Such subclasses have been investigated in various different contexts: super-Hopfian groups were investigated in \cite{GG}, torsion-free hereditarily separable groups were studied by various authors (see, for instance, \cite[Ch.VII. 4]{EM}), and more recently two of the current authors classified super-Bassian groups in \cite{DG}.

\medskip

First of all, note the following three simple properties of arbitrary classes.

\begin{proposition}\label{scprop} Let $\mathcal{X}, \mathcal{Y}$ be arbitrary classes of groups with associated classes $\mathcal{X}^s, \mathcal{X}_h, \mathcal{Y}^s, \mathcal{Y}_h$. The following three statements hold:
	
(i) if $\mathcal{X} \subseteq \mathcal{Y}$, then $\mathcal{X}^s \subseteq \mathcal{Y}^s$ and $\mathcal{X}_h \subseteq \mathcal{Y}_h$;
	
(ii) if $G \in \mathcal{X}^s$, then any homomorphic image of $G$ is also in $\mathcal{X}^s$ and any subgroup of $G$ is in $\mathcal{X}_h$;
	
(iii) $(\mathcal{X}^s)^s = \mathcal{X}^s$ and $(\mathcal{X}_h)_h = \mathcal{X}_h$.
\end{proposition}

\begin{proof} If $G \in \mathcal{X}^s$ and $X$ ia any epimorphic image of $G$, then $X \in \mathcal{X} \subseteq \mathcal{Y}$, so $G \in \mathcal{Y}^s$ giving (i). Part (ii) is immediate since compositions of epimorphisms are again epimorphisms. To establish part (iii), it suffices by part (i) to show that $\mathcal{X}^s \subseteq (\mathcal{X}^s)^s$, but this follows at once from part (ii).
	
The proofs for the hereditary properties are quite similar, and thus are left to the reader.	
\end{proof}

As mentioned in the introductory section, we concentrate on the class $\mathcal{DF}$ of directly finite groups. We will also be interested in the classes $\mathcal{B, H, RH, WH}$ which are all {\bf genuinely} subclasses of $\mathcal{DF}$ and have been intensively studied recently in \cite{CDG} and \cite{CDGK}.

\medskip

For the sake of completeness of the exposition, recall once again that these classes are defined as follows: a group $G$ is said to be Bassian ($G \in \B$) if $G$ cannot be embedded in a proper homomorphic image of itself, while $G$ is said to be Hopfian ($G \in \mathcal{H}$) if $G$ does not have a proper isomorphic quotient.

Our next two concepts are generalizations of Hopficity: a group $G$ is called {\it relatively Hopfian} ($G \in \RH$) if $G$ is not isomorphic to a proper direct summand of any of its proper quotients, while $G$ is called {\it weakly Hopfian} ($G \in \WH$) if $G$ is not isomorphic to a proper quotient $G/H$, where $H$ is a proper pure subgroup of $G$.

It was shown in \cite{CDGK} that $\mathcal{H} \subseteq \RH \subseteq \WH \subseteq \DF$ and also that each containment is strict. Likewise, it follows easily from the definition of a Bassian group stated above that $\mathcal{B}$ is strictly contained in $\mathcal{H}$.

More precisely, we wish to study in what follows the associated subclasses $\X ^s, \X _h$, where $\X$ is a class with $\B \subseteq \X \subseteq \DF$.

\section{Preliminaries}\label{prelim}

We begin this section with some results that will be helpful in our subsequent discussions. The first one, that is the Corner's Extension of Szele's Theorem - this is reference [U14] in \cite{G} -  has been used previously by one of the authors, but the full proof given by Corner has never been published heretofore; we correct this below. Recall that Szele's Theorem states that a basic subgroup of a $p$-group $A$ is an endomorphic image of $A$ - see, e.g., \cite[Theorem 36.1]{F1} or \cite[Theorem 5.6.10]{F}.

\begin{theorem}\label{corner} (Corner's Extension of Szele's Theorem)
Let $G$ be an extension of a torsion group $T$ by a countable torsion-free group, and let $B_p$ be a $p$-basic subgroup of the $p$-component of $T$. Then, $B_p$ is an endomorphic image of $G$.
\end{theorem}

\begin{proof} Consider firstly the situation in which $T$ is a $p$-group with basic subgroup $B$ and $G$ is an extension of $T$ by a countable torsion-free group. Embed $G$ in a group $G'$ such that $G'/T$ is a minimal divisible group containing $G/T$: if there exists a homomorphism $\omega: G' \to B$ such that $\omega(B)= B$, then we certainly have $\omega(G) = B$, and the theorem is proved. This means it will be enough to prove, under the additional hypothesis that $G/T$ be divisible, that there exists a map $\omega: G \to B$ such that $\omega(B) = B$.
	
Consider, then, an extension $G$ of the $p$-group $T$ such that $G/T$ is divisible, countable and torsion-free. Put $G_0 = G$, and as soon as $G'_{n-1}$ is defined, write $G'_{n-1} = B_n \oplus G'_n$, where $B_n$ is a maximal $p^n$-bounded direct summand of $G'_{n-1}$. Thus, by an obvious induction, for each $n$,
$$ G = B_1 \oplus B_2 \oplus \dots \oplus B_n + G'_n ;$$
and $B_1 \oplus B_2 \oplus \dots \oplus B_n$ is a maximal $p^n$-bounded direct summand of $T$, so $B = B_1 \oplus B_2 \oplus \dots $ is a basic subgroup of $T$. Now, $T/B$ is divisible, by the definition of a basic subgroup, and $G/T$ is divisible, by hypothesis; so, the exact sequence $0 \to T/B \to G/B \to G/T \to 0$ shows that $G/B$ is also divisible.	
	
If $g$ is an element of $G$, put $g'_0 = g$, and as soon as $g'_{n-1}$ is defined, write $g'_{n-1} = g_n + g'_n$, where $g_n \in B_n$ and $g'_n \in G'_n$; then, for each $n$, we have
$$g = g_1 + g_2 + \dots + g_n + g'_n.$$
In this way, $g$ determines a sequence $\{g_1, g_2 \dots \}$, which we shall further denote by $\pi g$ and construe as an element of the complete direct sum $P = \Sigma^*_n B_n$; clearly, $\pi$ is a homomorphism $G \to P$ which \lq\lq leaves $B$ invariant\rq\rq.
	
Suppose for the moment that we have a sequence of integers $s_n$ with the properties:
	
(i) $s_n \to \infty$,
	
(ii) $H(g_n) - s_n \to \infty$ for each $g \in G$.
	
Then, let $m$ be a positive integer such that $r(p^m B) = finrk(B)$, and take $E_{m+1}, E_{m+2}, \dots $ to be pairwise disjoint subsets of the set of integers $> m$ such that
	
(a) $r(\sum \limits_{n \in E_k} B_n) \geq r(B_k)$,
	
(b) $s_n \geq k \quad (n \in E_k).$
	
\medskip
	
Therefore, $(i)$ and the choice of $m$ ensure the existence of suitable $E_k$. Condition (a) allows us to define a homomorphism $\omega$ of $B$ onto itself which reduces to the identity on $B_1 \oplus B_2 \oplus \dots \oplus B_n$, maps $\sum\limits_{n \in E_k} B_n$ onto $B_k \ (k > m)$, and vanishes on $B_n$ if $n > m, n \notin \bigcup_k E_k$. Now, if $g \in B$, then all but a finite number of the $g_n$ vanish, and $g = \sum_ng_n$, so $\omega(g) = \sum_n\omega(g_n)$. We extend $\omega$ to the whole of $G$ by using the same formula to define $\omega(g)$ for general $g \in G$. To justify this definition, we must show that, for an arbitrary $g \in G$, all but a finite number of the $\omega(g_n)$ vanish, for then the sum $\sum_n \omega(g_n)$ has a meaning in $B$. But, if $g \in G$, then by (ii) there is an integer $N$ such that the height $H(g_n) \geq s_n$ whenever $n > N$, and we may suppose that $N \geq m$; if $n > N, n \notin \bigcup_k E_k$, then $\omega$ vanishes on $B_n$, so certainly $\omega(g_n) = 0$; and, if $n>N, n \in E_k$, then $H(\omega(g_n))\geq H(g_n) \geq s_n \geq k$, and since $\omega(g_n) \in B_k$, which is a pure subgroup of $B$ containing no non-zero element of height $\geq k$, it must be that $\omega(g_n) = 0$.
	
It remains for us to construct a sequence of integers $s_n$ satisfying points (i) and (ii). We remark first that, if $g \in G$, then the sequence of heights $H(g_n)$ tends to infinity. For, if $k$ is an arbitrary positive integer, then since $G/B$ is divisible, we may write $g = b + p^ng' \ (b \in B, g' \in G)$; but $b \in B_1 \oplus \dots \oplus B_n$ for some integer $N$, and for $n \geq N$ we infer $g_n = p^kg'_n$, so $H(g_n) \geq k$, which proves the assertion. Now, $G/T$ is countable, by hypothesis, so we denote by $g^{(1)}, g^{(2)}, \dots $ a complete set of representatives of the co-sets of $T$ in $G$. Put $N_0 = \{0\}$, and for $k = 1,2, \dots $ let $N_k$ be the least integer greater than $2k + 1$ such that $H(g_n^{(i)}) \geq 2k$ for $i = 1, \dots , k$ and all $n \geq N_k$. Define $s_n = k$ for $N_{k-1} < n \leq N_k$. Now, it is evident that $N_k \to \infty$ monotonically, so $s_n$ is defined for every positive integer $n$; and, moreover, $s_n \geq k$ for $n > N_{k-1}$, so (i) holds. If $N_{k-1} < n \leq N_k$, then $n \geq 1 + N_{k-1} \geq 2k = 2s_n$; so $n \geq 2s_n$ for all $n$. 	
	
To verify (ii), we note that, if $g \in G$, then $g = t + g^{(i)}$ for some $t \in T$ and $i = 1, 2, \dots$; but $t$ is a torsion element, so $p^mt = 0$ for some $m$ and, therefore, $H(t_n) \geq n - m \geq 2s_n - m$ for $n = 1, 2, \dots $; and, if $n > N_i$, then $N_k < n \leq N_{k+1}$ for some $k \geq i$, so $H(g_n^{(i)}) \geq 2k = 2s_n - 2$; consequently, for $n > N_i$, we have $H(g_n) \geq \min(2s_n - m, 2s_n -2)$, so $H(g_n) - s_n \geq s_n - \max(m, 2)$, and this tends to infinity, by (i). This completes the proof in the case where $T$ is a $p$-group.
	
If $T$ is now an arbitrary torsion group and $B_p$ is a basic subgroup of the $p$-component $T'$ of $T$, write $T = T' \oplus T''$. Then, the torsion subgroup of the quotient $G/T''$ is obviously $T/T''$, which is isomorphic with $T'$, so its basic subgroup is isomorphic with $B_p$. Furthermore, $(G/T'')/(T/T'') \cong G/T$ is countable, so by the first part of the argument, $B_p$ is a homomorphic image of $G/T''$. Hence, $B_p$ is an endomorph of $G$, as claimed.
\end{proof}

\begin{remark} Although we shall make no use of it here, Corner also pointed out a further simple corollary: {\it let $B$ be a basic subgroup of a torsion group $T$ whose reduced part has only a finite number of non-vanishing primary components, and let $G$ be an extension of $T$ by a countable torsion-free group. Then, $B$ is an endomorph of $G$}.
\end{remark}

The next result is, probably, well-known, but we include its proof for the sake of completeness. To avoid possible confusion with the usual $p$-rank, $r_p(G) = \dim G[p]$, we just write $basrank(A) = dim(A/pA)$, where $A/pA$ is considered as a vector space over the field of $p$ elements; our terminology is motivated by the fact that, for a $p$-group $G$, $basrank(G)$ is equal to $r_p(B)$ where $B$ is a basic subgroup of $G$. Note, however, that $basrank(G)$ may be finite even when $G$ is uncountable; for example, if $J_p$ is the usual group of $p$-adic integers, then $basrank(J_p) = 1$.

\begin{proposition}\label{useful}
Let $G$ be a torsion-free group of finite rank $n$ and let $X$ be an arbitrary epimorphic image of $G$. Then, $X$ is isomorphic to a subgroup of $\Q^{(n)} \oplus D$, where $D$ is a torsion divisible group, $D = \bigoplus\limits_{p \in \mathbb{P}} D_p$, where each $D_p$ is a direct sum of $\kappa_p$ copies of $\qc$ with $\kappa_p \leq n$ for all $p$. Equivalently, $X$ is a subgroup of $(\Q \oplus \Q/\Z)^{(n)}$.
\end{proposition}	

\begin{proof}
Let $\phi: G \twoheadrightarrow X$ be an epimorphism. Since $r_0(G) = n$, we obtain that $G \leq \Q^{(n)}$. Let $d(X)$ be a minimal divisible hull of $X$. Thus, the mapping $\phi$ extends to a homomorphism $\eta$ from $\Q^{(n)}$ into $d(X)$. But then
$$d(X) \geq \eta(\Q^{(n)}) \geq \eta(G) = \phi (G) = X,$$ and since the image of $\Q^{(n)}$ under $\eta$ is divisible, the minimality of $d(X)$ gives $d(X) = \eta(\Q^{(n)})$, i.e., $d(X)$ is an epimorphic image of $\Q^{(n)}$. So, $d(X)$ has the form $d(X) = D \oplus D_1$, where $D_1$ is torsion-free divisible and $D$ is torsion divisible. Apparently, $r_0(D_1) \leq n$ since $D_1$ is an epic image of $\Q^{(n)}$.
	
Now, $D$ can be expressed as $D = \bigoplus\limits_{p \in \mathbb{P}}D_p$, where each $D_p$ has the form $D_p = \bigoplus\limits_{\kappa_p}\qc$. Moreover, $D$ is an epic image of $\Q^{(n)}$, so we have an exact sequence $0 \to K \to \Q^{(n)} \to D \to 0$. With the aid of a standard result - see, e.g., \cite[Theorem 0.2]{A} - we get, for each $p$,
	
$$\dim(D[p]) + basrank(\Q^{(n)}) = basrank(K) + basrank(D).$$
	
Furthermore, we deduce that $\dim(D[p]) = \dim(D_p[p]) = \kappa_p$, while each of the divisible terms $D$ and $\Q^{(n)}$ clearly has zero $basrank$. It follows that $\kappa_p$ equals the $basrank$ of $K$. Since $K$ is torsion-free of finite rank, its $basrank$ is at most $r_0(K)$ - see \cite[Theorem 0.1]{A} - which, in turn, is bounded by $n$, as required.
\end{proof}

The next technicality is crucial.

\begin{lemma}\label{epicfin} Let $\alpha: M \twoheadrightarrow Y$ be an epimorphism from a reduced Bassian group $M$ to a $p$-group $Y$. Then $r_p(Y)$ is finite.
\end{lemma}

\begin{proof} Let the torsion subgroup of $M$ be expressed as $T = M_p \oplus \bigoplus\limits_{q \neq p} M_q = M_p \oplus M_{p'}$, where the various groups $M_q$ are the $q$-primary components of $M$. Notice that, for each $M_q$, $r_q(M_q)$ is finite rank since $M$ is, by hypothesis, Bassian and since $M$ is also reduced, then each primary component is actually finite. Since $\alpha(M_{p'})$ is necessarily zero, we get an induced epimorphism $\bar{\alpha}: M/M_{p'} \twoheadrightarrow Y$. Note also that the sequence
$$0 \to T/M_{p'} \to M/M_{p'} \to M/T \to 0$$ is pure-exact, and since the first term is finite, being isomorphic to $M_p$, we get a decomposition $M/M_{p'} = F \oplus L$, where $F \cong T/M_{p'}$ is finite and $L \cong M/T$ is torsion-free of finite rank, as $M$ is Bassian.
	
So, $Y = \bar{\alpha}(M/M_{p'}) = \bar{\alpha}(F) + \bar{\alpha}(L)$; now, the image of $F$ is finite and the image of $L$ is a homomorphic image of a torsion-free group of finite rank. Thus, thanks to Proposition \ref{useful} above, $\bar{\alpha}(L)$ is isomorphic to a subgroup a group $D_p$ of the form $D_p = \bigoplus\limits_{\kappa_p}\qc$ with $\kappa_p \leq r_0(L)$. Hence, $r_p(Y)$ is finite, as asserted.
\end{proof}

Our final preliminary statement is the following.

\begin{proposition}\label{inftf}
If $A$ is a torsion-free group, then $A$ has a homomorphic image of the form $\bigoplus\limits_{\aleph_0}\qc$ (for any prime $p$) if, and only if, $A$ is of infinite rank.
\end{proposition} 	

\begin{proof} If $A$ has a divisible direct summand of infinite rank, then the result is immediate. Clearly, then it suffices to establish the result under the additional hypothesis that $A$ is reduced. Set $D := \bigoplus\limits_{\aleph_0}\qc$.
	
Assume that $A$ is of infinite rank, so it has a free subgroup $F$ of countably infinite rank and as $D$ is countably generated, it follows from the standard property of freeness that there is an epimorphism $\phi: F \twoheadrightarrow D$. Since $D$ is injective, $\phi$ lifts to a homomorphism $\eta: A \to D$. Hence, $D = \phi(F) = \eta(F) \leq \eta(A) \leq D$, so that $D$ is an epic image of $A$, as expected.
	
Conversely, if $A$ has finite rank, it follows from Proposition~\ref{useful} that if $D_{\kappa} = \bigoplus\limits_{\kappa} \Z(p^\infty)$ is an epimorphic image of $A$, then $\kappa$ is finite, as needed.
\end{proof}

\section{Hereditarily Directly Finite Groups}\label{hdfg}

In this section, we introduce a new class of directly finite groups which will play a key role in our subsequent discussions.

A group $G$ is said to be an {\bf Extended Bassian} group if $G$ has the form $G = B \oplus D$, where $B$ is Bassian and $D$ is a torsion divisible group of the form $D = \bigoplus\limits_p\qc^{(n_p)}$, where each $n_p$ is finite. We denote the class of Extended Bassian groups by $\mathcal{EB}$; note that $\B \subseteq \EB$; clearly, $\EB \nsubseteq \mathcal{H}$ since no quasi-cyclic group $\qc$ is Hopfian, but we shall show below that $\EB \subseteq \RH$. The containments are shown in the diagram below.

\begin{proposition}\label{ebrhop} The class of groups $\mathcal{EB}$ is contained in the class $\mathcal{RH}$ of relatively Hopfian groups.
\end{proposition}

\begin{proof} Suppose then that $G = B \oplus D$ is in $\mathcal{EB}$ with $D$ having a primary decomposition $D = \bigoplus\limits_p D_p$. Now, $t(G) = t(B) \oplus D$, so that $G/t(G) \cong B/t(B)$ and it follows from \cite[Proposition 3.11]{CDG} that $B/t(B)$ is torsion-free of finite rank. Therefore, $G/t(G)$ is torsion-free of finite rank and hence so is Hopfian. However, since $B$ is Bassian, $t_p(B)$ is a finite group, and as $t_p(G) = t_p(B) \oplus D_p$, it follows from \cite[Corollary 2.13]{CDGK} that $t_p(G)$ will be relatively Hopfian if and only if $D_p$ is relatively Hopfian for all primes $p$; the latter statement is true by \cite[Proposition 2.17]{CDGK}. According to \cite[Proposition 2.14]{CDGK}, we conclude that $t_p(G)$ is relatively Hopfian. Since $t(G)$ is clearly fully invariant in $G$, \cite[Theorem 4.1]{CDGK} gives us that $G$ is relatively Hopfian. Since $G$ was arbitrary in $\mathcal{EB}$, we detect that $\mathcal{EB} \subseteq \mathcal{RH}$, as promised.
\end{proof}

\bigskip
\bigskip

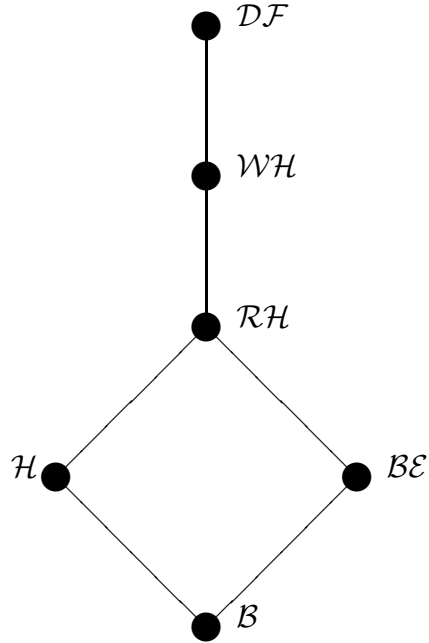
\setlength{\unitlength}{2mm}
\begin{picture}(40,50)(-10,0)
	\drawline(20,50)(20,40)
	\drawline(20,40)(20,30)
	\drawline(20,30)(10,20)
	\drawline(20,30)(30,20)
	\drawline(10,20)(20,10)
	\drawline(30,20)(20,10)
	%\drawline(20,10)(20,0)
	\put(20,50){\circle*{2}}
	\put(20,40){\circle*{2}}
	\put(20,30){\circle*{2}}
	\put(10,20){\circle*{2}}
	\put(30,20){\circle*{2}}
	\put(20,10){\circle*{2}}
	%\put(20,0){\circle*{2}}
	\put(22,50){$\mathcal{DF}$}
	\put(22,40){$\mathcal{WH}$}
	\put(22,30){$\mathcal{RH}$}
	\put(7,20){$\mathcal{H}$}
	\put(32,20){$\mathcal{BE}$}
	\put(22,10){$\mathcal{B}$}
	%\put(22,0){$(0,1,2,\infty,\ldots)$}
\end{picture}

\centerline{{ Containment Relationships}}

\bigskip

It was noted in \cite{CDG} that subgroups of Bassian groups are again Bassian and our next result shows that the class $\EB$ enjoys the same critical property.

\begin{proposition}\label{ebebh} If $G \in \EB$, then every subgroup of $G$ is again in $\EB$; equivalently, $\EB = \EB_h$.
\end{proposition}

\begin{proof}
Suppose that $G = B \oplus D$, where $B$ is Bassian and $D$ is torsion divisible with each $p$-primary component of finite rank. Let $H$ be an arbitrary subgroup of $G$ and write $H = D_0 \oplus H_0$, where $D_0$ is the maximal torsion divisible subgroup of $H$. Then, $D_0 \leq D$ and so $D_0$ has the form $D_0 = \bigoplus\limits_p\qc^{(k_p)}$, where each $k_p$ is finite. However, $$H_0 \cong H/D_0 = H/(H \cap D) \cong (H + D)/D \leq G/D \cong B,$$ and thus $H_0$ is isomorphic to a subgroup of $B$, whence it is Bassian and $H$ is then Extended Bassian. Hence, if $G \in \EB$, then $G \in \EB_h$ and since the latter is contained in $\EB$, we have equality $\EB = \EB_h$, as desired.
\end{proof}

Suppose now that $G$ is an arbitrary group in $\DF_h$. Clearly, $r_0(G)$ is finite; otherwise, $G$ would have a free subgroup of infinite rank which is not in $\DF$. Write $G = D \oplus G'$, where $D$ is the maximal torsion divisible subgroup of $G$. Then, $D = \bigoplus\limits_p\qc^{(k_p)}$ for some $k_p$, and the direct finiteness of $D$ assures that each $k_p$ is finite. Now, if $T_p$ stands for the $p$-component of the torsion subgroup of $G'$, then $T_p[p]$ is again in $\DF$, and thus the socle must have finite dimension. Since $G'$ has no torsion divisible subgroup, we conclude that each $T_p$ is actually finite.

Now, write $G' = \Q^{(k_0)} \oplus G''$, where $G''$ has no torsion-free divisible subgroup; note that $G''$ is then reduced. Since $k_0 \leq r_0(G') = r_0(G)$, clearly $k_0$ is finite and $r_0(G'')$ is also finite. Thus, $G''$ is a reduced group with finite $p$-primary subgroups and finite torsion-free rank, and so it follows from the classification of Bassian groups obtained in \cite[Main Theorem]{CDG} that $G''$ is Bassian. As $k_0$ is also finite, the same classification yields that $G'$ is Bassian. Consequently, we have established the following statement:

\begin{proposition}\label{dfhebh} If $G \in \DF_h$, then $G \in \EB = \EB_h$.
\end{proposition}

This simple proposition enables us to give a complete description of hereditarily $\X$ classes for many subclasses of $\DF$. Concretely, we are prepared to prove the following main assertion:

\begin{theorem}\label{classher} (i) If $\EB \subseteq \X \subseteq \DF$, then $\X_h = \EB$;

(ii) If $\B \subseteq \X \subseteq \mathcal{H}$, then $\X_h = \B$. \end{theorem}

\begin{proof} (i) From our hypothesis and Proposition \ref{scprop} it must be that $\EB_h \subseteq \X_h \subseteq \DF_h$, and thus it follows from Proposition \ref{dfhebh} for the last term that $\DF_h \subseteq \EB_h$. So, $\X)h = \EB_h = \EB$.
	
(ii) If $G \in \X_h$, then as $\X \subseteq \mathcal{H} \subseteq \DF$, we have $G \in \DF_h = \EB$. Now, if $G = B \oplus D$, with $B$ Bassian and $D$ torsion divisible, then $G \in \mathcal{H}$ implies that $D = \{0\}$, and hence $G \in \B = \B_h$. Thus, $\B_h \subseteq \X_h \subseteq \B_h = \B$, and then $\X_h = \B$, as wanted.
\end{proof} 	

\begin{remark} Part(ii) of Theorem \ref{classher} could be derived directly from the observation in \cite[Proposition 2.1.]{DG} that the class of Bassian groups coincides with the class of hereditarily Hopfian groups, $\B = \mathcal{H}_h$; nevertheless, we preferred to use the general setting of Proposition~\ref{dfhebh} to get a more conceptual proof.
\end{remark}

\section{Super Directly Finite Groups}\label{sdfg}

We now investigate the groups which are super directly finite.

\subsection{Divisible Groups}

Letting $D$ be an arbitrary divisible group, then $D$ has the form $D = \bigoplus \Q \oplus \bigoplus\limits_{p \in \mathcal{P}} \qc$, so that $r_0(D)$ determines the number of copies of $\Q$, while for each prime $p$, $r_p(D)$ determines the number of copies of $\qc$ (see \cite{F1,F}). The next result characterizes the divisible groups which are super directly finite.

\begin{proposition}\label{sdfdiv} A divisible group $D$ is super directly finite if, and only if, each of the ranks $r_0(D), r_p(D)$ is finite.
\end{proposition}

\begin{proof} The necessity is evident since if any of the rank invariants is infinite $D$ will have an isomorphic proper direct summand.
	
Reciprocally, suppose that each of the ranks $r_0(D), r_p(D)$ is finite; write $D = F \oplus T$, where $T$ is the torsion subgroup of $D$ and $F$ is torsion-free divisible. Thus, utilizing \cite[Proposition 2.17(b)]{CDGK}, $D$ is directly finite, and if $\phi :G \twoheadrightarrow X$ is an epimorphism, then $X = D_1 \oplus Y$, where $D_1 = \phi(F)$ and $Y = \phi(T)$. Furthermore, an application of Proposition~\ref{useful} above forces the result that the torsion-free rank $r_0(D)$ of $D_1$ and each of its ranks $r_p(D)$ are finite. Employing \cite[Proposition 2.15]{CDGK} to $Y$, we again get that each rank $r_p(Y)$ is finite, so that the ranks $r_p(X)$ and its torsion-free rank are all finite; so, Proposition 2.17 in \cite{CDGK} guarantees the pursued result that $X$ is directly finite. Since $X$ was an arbitrary epimorphic image of $D$, we infer that $D$ is super directly finite, as formulated.
\end{proof}

Our next result is a simplified version of \cite[Theorem 2.6]{CDGK} adequate for our current requirements.

\begin{proposition}\label{redplusdiv} Let $G = D \oplus R$ be a group, where $D$ is divisible and $R$ is reduced. Then, $G$ is super directly finite if, and only if, both $D$ and $R$ are super directly finite.
\end{proposition}

\begin{proof} The necessity is automatically true since any super class is closed under homomorphic images and both $D$ and $R$ are such.
	
Oppositely, suppose that $D$ and $R$ are both super directly finite, and let $X$ be an arbitrary epimorphic image of $G$ via the map $\phi$. Now, $\phi(D)$ is divisible and so $X = \phi(D) \oplus Y$ for some $Y$. Consider the composition $\psi : G \twoheadrightarrow X \twoheadrightarrow G/\phi(D)$. Therefore, $G/\ker \psi \cong X/\phi(D)$, so that $Y$ is an epimorphic image of $G/\Ker \psi$. Since $\Ker \psi$ contains $D$, $G/\Ker \psi $ is itself an epic image of $G/D$, so that $Y$ is an epimorphic image of $R$ and it is thus directly finite.
	
Furthermore, by assumption, $D$ is super directly finite and so it follows from Proposition \ref{sdfdiv} above that all the rank invariants of $D$ are finite. Then, arguing as in the proof of that proposition, and invoking Proposition 2.15 of \cite{CDGK} in combination with Proposition~\ref{useful}, we see that all the rank invariants of $\phi(D)$ are again finite. Appealing to \cite[Proposition 2.17]{CDGK}, we have that $X = \phi(D) \oplus Y$ is directly finite. Since $X$ was an arbitrary epimorphic image of $G$, we finally established that $G$ is super directly finite, as stated.  	
\end{proof}

\subsection{Reduced Groups}

Throughout this subsection, we will assume that we are dealing with reduced groups which are directly finite.

Our first result is a simple consequence of Proposition \ref{inftf} above.

\begin{proposition}\label{finrk} Let $G$ be a reduced torsion-free group in $\mathcal{DF}^s$. Then, $r_0(G)$ is finite.
\end{proposition}

\begin{proof} If not, the quotient $G/tG$ is in $\mathcal{DF}$ and, moreover, is torsion-free of infinite rank. Owing to Proposition \ref{inftf}, we see that $G/tG$, and hence $G$ itself, has an epimorphic image of the form $\bigoplus\limits_{\aleph_0}\qc$, but such a group is obviously not directly finite - contradiction.
\end{proof}

Our next result provides useful information about the ranks of primary components of groups in $\mathcal{DF}^s$.

\begin{proposition}\label{nottf} Let $G$ be a reduced group in $\mathcal{DF}^s$ which is not torsion-free. Then, each $p$-primary component of $G$ is finite.
\end{proposition}

\begin{proof}
We split the argument into two cases: (a) $G$ is torsion, and (b) $G$ is mixed.

\medskip
	
\noindent{\bf Case (a):} Let $G = \bigoplus\limits_{p \in \P}T_p$ be the usual primary decomposition of $G$. Suppose, for a contradiction, that some primary component $T_p$ is infinite and let $B$ be a basic subgroup of $T_p$. Thus, $B$ is an endomorphic image of $T_p$ by the quoted above Szele's Theorem, and so $B \in \mathcal{DF}$. If $B$ is infinite, then it either has an infinite homocyclic component, which is necessarily a summand of $B$, or it is unbounded. As $B$ is directly finite, the first situation manifestly cannot occur. However, if $B$ is unbounded, it has an epimorphic image $B/pB$ which again must be directly finite - impossible, since in this situation $B/pB$ is elementary of infinite rank.

\medskip
	
\noindent{\bf Case(b):} Assume now that $G$ is a reduced mixed group with torsion subgroup $tG = 	 \bigoplus\limits_{p \in \P}T_p$. It follows from Proposition \ref{finrk} that $r_0(G)$ is finite, and thus $G$ is an extension of the torsion group $tG$ by a torsion-free group $G/tG$ which is countable since $r_0(G/tG) = r_o(G)$ is finite. Hence the aforementioned Corner's Theorem \ref{corner} is applicable and, if $B_p$ is a basic subgroup of $T_p$, then $B_p$ is an endomorphic image of $G$, and so $B_p \in \mathcal{DF}^s$. The argument in case (a) shows that $B_p$ must be finite, and since it is basic in $T_p$, the latter must also be finite, as required.
\end{proof}

As a direct consequence, we extract:

\begin{corollary}\label{redbass} If $G \in \mathcal{DF}^s$ is reduced, then $G$ is Bassian.
\end{corollary}

\begin{proof} It follows from the two propositions alluded to above that each of the ranks $r_0(G)$ and $r_p(G) (p \in \P)$ is finite, and the result then follows from the Main Theorem of \cite{CDG}.
\end{proof}

\subsection{General Case}

In this subsection, we combine our results from the two previous sections to investigate super directly finite groups which have both divisible and reduced components. To that target, let $G$ be a super directly finite group of the form $G = D \oplus R$, where $D$ is a non-zero divisible group and $R$ is a non-trivial reduced group; it follows from Proposition \ref{redplusdiv} that both $D$ and $R$ are super directly finite, and so Corollary \ref{redbass} applies to get that $R$ is Bassian. Applying Proposition \ref{finrk} to $D$, we see that $D$ is of the form $D = \Q^{(n_0)} \oplus \bigoplus\limits_p \qc^{(n_p)}$ with $n_0, n_p$ all finite. Set $B := \Q^{(n_0)} \oplus R$ and note that it follows from the classification of Bassian groups given in \cite[Main Theorem (ii)]{CDG} that $B$ is Bassian. Thus, any super directly finite group is of the form
$$G = B \oplus D, \ {\rm with} \ B \ {\rm Bassian \ and} \ D = \bigoplus\limits_p \qc^{(n_p)}, \ n_p \ {\rm finite \ for \ all} \ p.   \quad (*) $$

Let us move now to the more general situation of a class of groups $\mathcal{X}$ contained in the class of directly finite groups $\mathcal{DF}$. We see two different situations arising, depending on where $\mathcal{X}$ is situated relative to the classes $\mathcal{H}$ and $\mathcal{EB}$.

Suppose now that $\mathcal{X}$ is any class satisfying $\mathcal{B} \subseteq \mathcal{X} \subseteq \mathcal{H}$. Since $\mathcal{X} \subseteq \mathcal{H}$, we derive $\mathcal{X}^s \subseteq \mathcal{H}^s \subseteq \mathcal{DF}^s$. So, if $G \in \mathcal{X}^s$, then $G$ has the form $G = B \oplus D$ as in $(*)$ above. But $G$ is Hopfian, so $D = \{0\}$ and hence $G$ is Bassian, and thus $\mathcal{X}^s \subseteq \mathcal{B}$. Since we also have that $\mathcal{B} \subseteq \mathcal{X}$, it follows that $$\mathcal{B}^s \subseteq \mathcal{X}^s = (\mathcal{X}^s)^s \subseteq \mathcal{B}^s,$$ so that $\mathcal{X}^s = \mathcal{B}^s$; in particular, $\mathcal{H}^s = \mathcal{B}^s$. So, we have established our next major statement:

\begin{theorem}\label{bxhcase} If $\B \subseteq \X \subseteq \mathcal{H}$, then a group $G$ is in $\X^s$ if, and only if, it is super Bassian. In particular, a group is super Hopfian if, and only if, it is super Bassian.
\end{theorem}

Recall that, in \cite{DG}, a group $G$ was said to satisfy {\it property} $( \mathcal{P})$ if no quasi-cyclic group is an epimorphic image of $G$. Handling this notion, a classification of super Bassian groups was given in \cite[Theorem 2.5]{DG}.

Specifically, a torsion group is super Bassian if, and only if, it is reduced Bassian and so a torsion group is super Bassian if, and only if, each of its primary components is finite; in the other vein, a torsion-free group is super Bassian if, and only if, it is Bassian and has property $(\mathcal{P})$. Thus, a torsion-free group $G$ is super Bassian if, and only if, it has finite rank and, for every prime $p$, the localization $G_{(p)}$ is a free $\Z_{(p)}$- module. (For further equivalent conditions relating to property $(\mathcal{P})$ we refer to the discussion at the end of Section 2 in \cite{DG}.) For mixed groups $G$, a similar characterization in terms of property $(\mathcal{P})$ holds: in fact, $G$ is super Bassian if, and only if, it is Bassian and has property $(\mathcal{P})$.

\medskip

Suppose now that $\mathcal{X}$ is any class such that $\mathcal{EB} \subseteq \mathcal{X} \subseteq \mathcal{DF}$.
Our objective is to give a concrete description of all such super classes $\mathcal{X}^s$ in this setting. The next result is key.

\begin{proposition}\label{key} Suppose $G = B \oplus D$, where $B$ is a Bassian group and $D = \bigoplus\limits_p \qc^{(n_p)}$ with each $n_p$ finite. Then, any epimorphic image $X$ of $G$ is again of the form $X = B_1 \oplus D_1$, where $B_1$ is Bassian and $D_1 = \bigoplus\limits_p \qc^{(m_p)}$ with each $m_p$ finite. Equivalently, $\EB = \EB^s$. \end{proposition}

\begin{proof} Let $B = A \oplus \Q^{(n)}$, where $A$ is a reduced Bassian group and $n$ is finite; notice that, in accordance with the classification of Bassian groups given in \cite[Main Theorem]{CDG}, this is possible. Put $K := \Q^{(n)} \oplus D$ and let $\phi$ be an arbitrary, but fixed, epimorphism $\phi: G \twoheadrightarrow X$. Thus, $\phi(K)$ has the form $\Q^{(k)} \oplus \bigoplus\limits_p \qc^{(k_p)}$, where $k, k_p$ are all finite. Hence, $X = \phi(K) \oplus Y$, where $$Y \cong X/\phi(K) = (\phi(A) + \phi(K))/\phi(K) \cong \phi(A)/(\phi(A) \cap \phi(K));$$ in particular, $Y$ is an epimorphic image of the reduced Bassian group $A$.
	
Let $Y = \Q^{(t)} \oplus Z$, where $Z$ has no torsion-free divisible subgroups; note that $t$ is finite since $Y$, and hence $\Q^{(t)}$, is an epimorphic image of $G$ which has finite torsion-free rank.
	
Now, consider $Z$ and, for each prime $p$, let $Z_p$ denote the $p$-torsion subgroup of $Z$; notice that $Z$ is again an epimorphic image of the reduced Bassian group $A$, say $Z = \alpha (A)$ for some epimorphism $\alpha$. Since $\alpha$ is an epimorphism, for each prime $p$, there is a subgroup, $A^p$ say, such that $\alpha(A^p) = Z_p$. It is readily seen that $A^p$ is a subgroup of the reduced Bassian group $A$, and it follows from the classification of reduced Bassian groups proven in \cite[Main Theorem]{CDG} that $A^p$ itself is too reduced Bassian. Consulting with Lemma \ref{epicfin}, we see that $r_p(Z_p)$ is finite for all primes $p$, and so $r_p(Z)$ is finite for all primes $p$.
	
Now, let $Z = \bigoplus\limits_p \qc^{(t_p)} \oplus W$, where $W$ has no torsion divisible components whence $W$ is reduced since it is also a direct summand of the group $Z$, which has no torsion-free divisible subgroups. Furthermore, as $r_p(Z)$ is finite, we deduce that $t_p$ is finite and $r_p(W)$ is also finite for all primes $p$; observe too that $r_0(W)$ must be finite as it is an epic image of $G$ which has finite torsion-free rank.
	
Collecting terms, we see that $$X = \Q^{(k)} \oplus \bigoplus\limits_p \qc^{(k_p)} \oplus \Q^{(t)} \oplus \bigoplus\limits_p \qc^{(t_p)} \oplus W,$$  	
so that one easily checks that $X = (\Q^{(k+t)} \oplus W) \oplus \bigoplus\limits_p \qc^{(k_p + t_p)}$. But, $W$ is reduced and all of the ranks $r_p(W), r_0(W)$ are finite, so it again follows from \cite[Main Theorem (i)]{CDG} that $W$ is reduced Bassian; applying part (ii) of that same classification, we observe that $\Q^{(k+t)} \oplus W$ is Bassian. And finally, if we set $W := B_1$ and $D_1 := \bigoplus\limits_p \qc^{(k_p + t_p)}$, we get the desired result.
\end{proof}

Furthermore, since $\mathcal{DF}^s \subseteq \mathcal{EB}$, we see that
$$\mathcal{EB}^s \subseteq \mathcal{DF}^s = (\mathcal{DF}^s)^s \subseteq \mathcal{EB}^s \subseteq \mathcal{X}^s \subseteq \mathcal{DF}^s.$$ So, we find that $\mathcal{X}^s = \mathcal{DF}^s = \mathcal{EB}^s$.
However, Proposition \ref{key} above is just the statement that $\mathcal{EB}^s = \mathcal{EB}$ and so, after all, we have established the wanted classification.

\begin{theorem}\label{xebscase} If $\mathcal{EB} \subseteq \mathcal{X} \subseteq \mathcal{DF}$, then a group $G$ is in $\X^s$ if, and only if, it is in $\mathcal{EB}$.
\end{theorem}

%\begin{remark} It is perhaps worth noting that classes like $\mathcal{H}, \DF$ are, in fact classes in the set theoretical sense: for every infinite cardinal $\kappa$ there exist groups in $\mathcal{H}$ and $\DF$ of cardinality greater than $\kappa$. However, in both cases the corresponding super and hereditary  classes consist of countable groups, so they form a set. \end{remark}

\section{Concluding Remarks and Open Problems}

Exploiting the criterion proven in \cite{CDK} for a $p$-group to be directly finite, it plainly follows that a directly finite $p$-group having all subgroups Bassian must also be Bassian. Nevertheless, our results illustrated above considerably expand this helpful observation.

\medskip

On the other side, in the context of subclasses of arbitrary directly finite groups, it is clear that the imposition of closure under subgroups or homomorphic images is an extremely strong restriction: classes which are proper classes in the set-theoretical sense can actually be reduced to countable sets. Our current approach does not, however, handle every situation involving directly finite groups, so we close our work by posing the following pair of possible difficult questions.

\bigskip

\noindent{\bf Question 1.} Can the class $\X_h$ be completely characterized in the situation where $\X \subseteq \DF$ but $\X \nsubseteq \mathcal{H}$ and $\EB \nsubseteq \X$?

\medskip

\noindent{\bf Question 2.} Can the class $\X^s$ be completely characterized in the situation where $\X \subseteq \DF$ but $\X \nsubseteq \mathcal{H}$ and $\EB \nsubseteq \X$?

\end{document}